\documentclass[11pt]{amsart}
\usepackage{color, amssymb, amsmath, mathrsfs}
\usepackage{enumerate}

\usepackage[colorlinks]{hyperref}
\usepackage{graphicx}

\input xymatrix
\xyoption{all}

\hypersetup{linkcolor=blue, urlcolor=blue, citecolor=red}

\newtheorem{theorem}{Theorem}[section]
\newtheorem{proposition}[theorem]{Proposition}
\newtheorem{corollary}[theorem]{Corollary}
\newtheorem{lemma}[theorem]{Lemma}

\newtheorem{problem}[theorem]{Problem}
\newtheorem{claim}[theorem]{Claim}

\theoremstyle{definition}
\newtheorem{definition}[theorem]{Definition}

\newcommand{\F}{\mathcal{F}}

\newcommand{\Sch}{\operatorname{Sch}}
\newcommand{\w}{\omega}

\newcommand{\Ult}{\operatorname{Ult}}
\newcommand{\Ker}{\operatorname{Ker}}

\title{Schur ultrafilters and Bohr compactifications of topological groups}

\author{S. Bardyla, P. Zlato\v{s}}

\address{Serhii Bardyla: University of Vienna, Institute of Mathematics, Vienna, Austria.}
\email{sbardyla@gmail.com}

\address{Pavol Zlato\v{s}: Comenius University, Faculty of Mathematics, Physics and Informatics, Bratislava, Slovakia}
\email{zlatos@fmph.uniba.sk}
\makeatletter
\@namedef{subjclassname@2020}{%
  \textup{2020} Mathematics Subject Classification}
\makeatother


\thanks{The research of the first named author was funded in whole by the Austrian Science Fund FWF [10.55776/ESP399]. The research of the second named author was supported by grant no.~066UK--4/2023 of the Slovak grant agency KEGA}

\begin{document}
\begin{abstract}
In this paper, we investigate Schur ultrafilters on groups.
Using the algebraic structure of Stone-\v{C}ech compactifications of discrete groups,
we give a new description of Bohr compactifications of topological groups in terms of Schur ultrafilters.
This approach allows us to characterize which compact Hausdorff admissible right topological groups (briefly, chart groups)
are topological groups. Namely, a chart group $G$ is a topological group if and only if every Schur ultrafilter on $G$ converges to the unit of $G$.
\end{abstract}
\maketitle

\section{Introduction}
Ultrafilters on topological groups form a well-established topic in Topological Algebra with plenty of applications in Combinatorics and Topological Dynamics, see the monographs~\cite{HS, Zel} and references therein. Schur ultrafilters were introduced by Protasov~\cite{P}, as a generalization of idempotent ultrafilters.  
In this paper, we investigate Schur ultrafilters on groups and apply them to Bohr compactifications of topological groups and automatic continuity of group operations in chart groups.

Bohr compactifications of topological groups represent another widely studied topic,
see \cite{CH, GF, GF1, Gis1, Gis2, H}. Since the structure of Bohr compactification is rather non-trivial, the results providing its description are of particular interest.
For instance, Anzai and Kakutani~\cite{A} described
Bohr compactifications of locally compact commutative groups using Pontryagin duality.
A connection between Bohr compactifications and Chu duality was established
by Ferrer and Hern\'{a}ndez~\cite{FH}. 
Gismatullin,
Jagiella and Krupi\'nski~\cite{GJK} used model theory to describe Bohr compactifications of
certain matrix groups. Bohr compactifications of the arithmetic, lamplighter and Heisenberg groups were described by Bekka in~\cite{Bekka, Bekka1}.   Davey, Haviar and Priestley~\cite{DHP} studied Bohr
compactifications of fairly general algebraic structures within a category-theoretical framework.
Hart and Kunen~\cite{HK1,HK2} investigated Bohr compactifications of discrete structures, including groups, rings and semilattices.  
Yet another description of Bohr compactifications of discrete commutative groups was given in~\cite{Zlatos}, bringing to focus the algebraic structure of Stone-\v{Cech} compactifications and Schur ultrafilters. In the present paper, we develop the techniques from~\cite{Zlatos} and, using Schur ultrafilters, obtain a description of Bohr compactifications of all Hausdorff topological groups. In addition to Schur ultrafilters, a major role in this description is played by the {\em ultrafilter semigroup} of $G$, consisting of all ultrafilters convergent to the unit of the corresponding topological group $G$. The ultrafilter semigroups are of independent interest and have been studied by Zelenyuk in~\cite{Z1, Z2}.

Chart groups appear naturally in Topological Dynamics, as enveloping semigroups of particular dynamical systems, see~\cite{GM, GM1, GM3, GM4, GM5, GM6, Megre1}. One of the central problems in the theory of chart groups is when a chart group is a topological group. Namioka~\cite{N} showed that each metrizable chart group is a topological group. This result was later generalized by Moors and Namioka~\cite{MN} for first-countable chart groups, by Glasner and Megrelishvili~\cite{GM5} for Fr\'echet-Urysohn chart groups, and by Reznichenko~\cite{R0} for chart groups with countable tightness. The aforementioned problem is a partial case of a more general problem of automatic continuity of group operations, discussed in~\cite{BR, E, E1, FHW, Mil, M2, SS}. In this paper, using Schur ultrafilters, we find necessary and sufficient conditions for a chart group to be a topological group. 
Also, we introduce and describe 
a new universal compactification of right topological groups, which can be viewed as a counterpart of Bohr compactification in the class of chart groups.

This paper is organized as follows. In Section~\ref{sec1} we give necessary definitions and state the main results.  Section~\ref{sec2} is devoted to the combinatorial, algebraic, and topological properties of Schur ultrafilters. In Section~\ref{sec3} we discuss Bohr compactifications of topological groups. Chart groups and related concepts are studied in Section~\ref{sec4}.

\section{Preliminaries and the main results}\label{sec1}
In this paper, all topological spaces are assumed to be {\em Hausdorff}.
We identify the smallest infinite ordinal $\w$ with the set of all non-negative integers. 
For an equivalence relation $\rho$ on a set $X$ and $x\in X$ let $[x]_\rho=\{y\in X:(x,y)\in\rho\}$. Put $X/\rho=\{[x]_\rho: x\in X\}$. The map $h_\rho:X \rightarrow X/\rho$ given by $h_\rho(x)=[x]_\rho$ is called a {\em quotient map}. If the set $X$ is equipped with a topology, then,
by default, we assume that $X/\rho$ carries the {\em quotient topology}, i.e., the finest topology that makes the quotient map $h_\rho$ continuous. An equivalence relation $\rho$ on a topological space $X$ is called {\em closed} if $\rho$ is a closed subset of $X{\times}X$, equipped with the product topology.

For every element $s$ of a
semigroup $S$ the {\em right shift} $\rho_s$ is defined by $\rho_s(x)=xs$. Dually, for every $s\in S$ the {\em left shift} $\lambda_s$ is defined by $\lambda_s(x)=sx$. 
A semigroup $S$ endowed with a topology $\tau$ is called {\em right topological} if for each $s\in S$ the right shift $\rho_s$ is continuous.   
A right topological semigroup $S$ is called {\em semitopological} if for each $s\in S$ the left shift $\lambda_s$ is continuous. Following~\cite{HS}, for a right
topological semigroup $S$ the set $$\Lambda(S)=\{s\in S: \lambda_s \hbox{ is continuous}\}$$ is called a {\em topological center} of $S$. A group $G$ endowed with a topology is called a {\em topological group} if multiplication and inversion are continuous in $G$.

An equivalence relation $\rho$ on a semigroup $S$ is called a {\em congruence} if $(a_1,b_1)\in \rho$ and $(a_2,b_2)\in\rho$ imply $(a_1a_2,b_1b_2)\in\rho$ for any
$a_1,b_1,a_2,b_2 \in S$. The set $S/\rho$ endowed with the semigroup
operation defined by $[x]_\rho [y]_\rho = [xy]_\rho$ is called a {\em quotient semigroup}.
The quotient map $h_\rho: S\rightarrow S/\rho$
is referred to as the {\em quotient homomorphism}. For any homomorphism $f$ defined on a semigroup $S$ let $$\Ker(f)=\{(x,y)\in S{\times}S: f(x)=f(y)\}.$$ Clearly, $\Ker(f)$ is a congruence on $S$. A congruence $\rho$ on a semigroup $S$ is called a {\em group congruence} if $S/\rho$ is a group.
Since $S{\times}S$ is a closed congruence on $S$ and the intersection of any family of closed congruences is a closed congruence, for every relation $R \subseteq S \times S$ there exists the least
closed congruence on $S$ which contains $R$.    

The Stone-\v{C}ech compactification $\beta X$ of a discrete space $X$ is the set of all ultrafilters on $X$ endowed with a topology $\tau$ given by the base $\mathcal B=\{\langle A\rangle: A\subseteq X\}$, where $\langle A\rangle=\{u\in\beta X: A\in u\}$. Recall that each element $x\in X$ is identified with the principal ultrafilter $\{A\subseteq X: x\in A\}$. If $S$ is a discrete semigroup, then the semigroup operation on $S$ can be canonically lifted to a semigroup operation on $\beta S$ as follows: if $u,v\in\beta S$, then $uv$ is a filter generated by the base consisting of the sets $\bigcup_{x\in U}xV_x$, where $U\in u$ and $\{V_x:x\in U\}\subset v$ are arbitrary. Equivalently, $$A\in uv \iff \{
s\in S : \lambda_s^{-1}(A)\in v
\} \in u.$$  
By~\cite[Theorem~4.1]{HS}, the defined above semigroup operation on $\beta S$ is unique among those extending the operation of $S$ and satisfying the following two natural conditions:
\begin{itemize}
    \item[(i)] for each $u\in\beta S$ the right shift $\rho_u$ is continuous;
    \item[(ii)] for each $s\in S$ the left shift $\lambda_s$ is continuous.
\end{itemize}
Thus, for any discrete semigroup $S$, $\beta S$ endowed with the aforementioned operation is a compact right topological semigroup.  For more about the algebraic structure of $\beta S$ see the monograph~\cite{HS} and references therein. Compact right topological semigroups have been extensively studied in the literature; of particular relevance to this
paper are~\cite{DLS, J, Mi, MP, SS1}, including the monograph~\cite{Pym}. 

An element $e$ of a semigroup $S$ is called {\em idempotent} if $ee=e$. 
The following renowned theorem was proven in~\cite{E1}.
\begin{theorem}[Ellis]\label{classic}
Each compact right topological semigroup contains an idempotent.    
\end{theorem}

Since for every infinite discrete group $G$, the remainder $\beta G\setminus G$ is a closed subsemigroup of $\beta G$. Theorem~\ref{classic} implies the existence of a free idempotent ultrafilter on every infinite group $G$. It is easy to check that $u$ is an idempotent ultrafilter on a group $G$ if and only if for every $U\in u$ there exist $x\in U$ and $F_x\in u$ such that $xF_x\subseteq U$. The notion of a free idempotent ultrafilter can be naturally weakened as follows.

\begin{definition}
An ultrafilter $u$ on a group $G$ is called 
\begin{enumerate}[\rm (i)]
    \item {\em Schur} if for each $U\in u$ there exist $x,y\in U$ such that $xy\in U$;
    \item {\em infinitary Schur} if for each $U\in u$ there exist $x\in U$ and an infinite subset $V\subset U$ such that $xV\subseteq U$.
\end{enumerate}    
\end{definition} 

The set of all Schur ultrafilters on a group $G$ is denoted by $\Sch(G)$. If $G$ is a topological group, then by $G_d$ we denote the group $G$ endowed with the discrete topology. For a topological group $G$, let $\Ult(G)$ be the set of all ultrafilters on $G$ convergent to $1_G$. $\Ult(G)$ is a closed subsemigroup of $\beta G_d$ and is usually referred to as an {\em ultrafilter semigroup} of $G$.  The following two definitions are crucial for this paper.

\begin{definition}\label{def}
Let $G$ be a topological group. We shall consider the following congruences on the Stone-\v{C}ech compactification $\beta G_d$ of the discrete group $G_d$:

\begin{enumerate}[\rm(i)]
\item $\Theta$ is the least closed congruence such that $\{(u,1_G):u \hbox{ is idempotent}\}\subset \Theta$;
\item $\Xi$ is the least closed congruence such that $\{(u,1_G):u\in \Sch(G)\}\subset \Xi$;
\item $\Phi$ is the least closed congruence such that $$\{(u,1_G):u \hbox{ is idempotent}\}\cup \{(u,1_G): u\in\Ult(G)\}\subset \Phi;$$

\item $\Psi$ is the least closed congruence such that $$\{(u,1_G):u\in \Sch(G)\}\cup \{(u,1_G): u\in \Ult(G)\}\subset \Psi.$$
\end{enumerate}
\end{definition}

\begin{definition}\label{map}
For a topological group $G$ and a congruence $\rho$ on $\beta G_d$, let the map $\mathfrak h_\rho: G\rightarrow \beta G_d/\rho$ be defined by $\mathfrak h_\rho(x)=h_\rho(x)=[x]_\rho$.
\end{definition}

The {\em Bohr compactification} of a topological group $G$ is a pair $(\mathfrak b G, b)$, which consists of a compact topological group $\mathfrak b G$ and a continuous homomorphism $b: G\rightarrow \mathfrak b G$ such that $b(G)$ is dense in $\mathfrak b G$, and for each compact topological group $H$ and continuous homomorphism $f: G\rightarrow H$, there exists a continuous homomorphism $g: \mathfrak b G\rightarrow H$ such that $f=g\circ b$, i.e., the following diagram commutes: 

$$
\resizebox{0.18\textwidth}{!}
{\xymatrix{
G\ar[r]^b\ar[d]_f & \mathfrak b G\ar[ld]^g\\ H
}
}
$$

It is easy to see that the above property characterizes the Bohr compactification $\mathfrak b G$ up to a topological isomorphism.
The following description of
Bohr compactifications of discrete commutative groups was obtained in~\cite{Zlatos}.
It can be formulated in terms of the concepts introduced above as follows.

\begin{theorem}[Zlato\v{s}]\label{intro}
For any discrete commutative group $G$ the pair $(\beta G/\Xi, \mathfrak h_\Xi)$ is the Bohr compactification of $G$.   
\end{theorem}

In this paper, we develop the ideas from~\cite{Zlatos} and obtain the following
description of Bohr compactification of an arbitrary topological group.
\begin{theorem}\label{topgr}
For any topological group $G$ the pair $(\beta G_d/\Psi,\mathfrak h_\Psi)$ is the Bohr compactification of $G$.
\end{theorem}

Observe that if a group $G$ is discrete, then $\Ult(G)=\{1_G\}$ and, as such, $\Psi=\Xi$.
Theorem~\ref{topgr} implies the following generalization of Theorem~\ref{intro}.

\begin{theorem}\label{Bohr}
For any discrete group $G$, the pair $(\beta G/\Xi, \mathfrak h_\Xi)$ is the Bohr compactification of $G$.   
\end{theorem}

A right topological group $G$ is called {\em admissible} if the set $\Lambda(G)=\{x\in G: \lambda_x \hbox{ is continuous}\}$ is dense in $G$. Following~\cite{GM5}, a compact Hausdorff admissible right topological group is called a {\em chart} group. For more information on the structure of chart groups, see~\cite{FU, M3, N1}. 
A {\em universal chartification} of a right topological group $G$ is a pair $(\mathfrak c G, c)$ consisting of a chart group $\mathfrak c G$ and a continuous homomorphism $c: G\rightarrow \mathfrak c G$, which satisfy the following conditions:

\begin{enumerate}[\rm (i)]
\item $c(G)$ is dense in $\mathfrak c G$;
\item $c(G)\subseteq \Lambda(\mathfrak c G)$;
\item for each chart group $H$ and continuous homomorphism $f: G\rightarrow H$ such that $f(G)\subseteq \Lambda (H)$, there exists a continuous homomorphism $g: \mathfrak c G\rightarrow H$ such that $f= g\circ c$.
\end{enumerate}

The following theorem describes universal chartifications of right topological groups. 
\begin{theorem}\label{chartgr}
For any right topological group $G$ the pair $(\beta G_d/\Phi,\mathfrak h_\Phi)$ is the universal chartification of $G$.
\end{theorem}

Observe that if a group $G$ is discrete, then $\Phi=\Theta$. This way,
Theorem~\ref{chartgr} yields the following.

\begin{theorem}\label{chart}
For any discrete group $G$, the pair $(\beta G/\Theta, \mathfrak h_\Theta)$ is the universal chartification of $G$.   
\end{theorem}

As was mentioned in the introduction,  one of the central problems in the theory of chart groups is when a chart group is a topological group. The following result gives necessary and sufficient conditions for a chart group to be a topological group.

\begin{theorem}\label{main}
For a chart group $G$ the following conditions are equivalent:

\begin{enumerate}[\rm (i)]
    \item $G$ is a topological group;
    \item every Schur ultrafilter on $G$ converges to $1_G$;
   \item there exists a dense subgroup $H\subseteq \Lambda(G)$ such that each Schur ultrafilter on $H$ converges to $1_G$.
    \end{enumerate}
\end{theorem}

Theorem~\ref{topgr} is proven in Section~\ref{sec3}, after preliminary work done in Section~\ref{sec2}. Theorems~\ref{chartgr} and~\ref{main} are proven in Section~\ref{sec4}.

\section{Schur ultrafilters}\label{sec2}

In this section, groups are assumed to be discrete and infinite, unless stated otherwise. We start with the algebraic properties of Schur ultrafilters. The item (ii) of the following result is proven in~\cite[Lemma 5.3]{P}. For the sake of completeness, we include its brief proof.

\begin{lemma}\label{newrev}
Let $G$ be a group, $u\in \Sch(G)$ and $v\in \beta G$. Then
\begin{enumerate}[\rm(i)]
 \item if $U\in uv$, then $UU^{-1}\in u$;
    \item if $U\in vu$, then $U^{-1}U\in u$.
\end{enumerate}
\end{lemma}

\begin{proof}
(i) Fix an arbitrary $B\in u$ and $U\in uv$. There exist $F\in u$ and $\{V_x:x\in F\}\subset v$ such that $F\subseteq B$ and $\bigcup_{x\in F}xV_x\subseteq U$. As $u$ is a Schur ultrafilter, there exist elements $a,b,c\in F$ such that $ab=c$. Put $V=V_a\cap V_b\cap V_{c}$ and fix $t\in V$. Note that $bt\in bV\subseteq bV_b\subseteq U$, $ct\in cV\subseteq cV_c\subseteq U$ and $a=cb^{-1}=ct(bt)^{-1}\in UU^{-1}$. It follows that $UU^{-1}\cap B\neq \emptyset$. Since $B$ is arbitrarily chosen, $UU^{-1}\in u$.

(ii) Fix an arbitrary $B\in u$ and $U\in vu$. There exist $x\in G$ and $F\in u$ such that $F\subseteq B$ and $xF\subseteq U$. As $u$ is a Schur ultrafilter, there exist elements $a,b,c\in F$ such that $ab=c$. Note that $xa\in U$, $xc\in U$ and $b=a^{-1}c=(xa)^{-1}xc\in U^{-1}U$. It follows that $U^{-1}U\cap B\neq \emptyset$. Since $B$ is arbitrarily chosen, $U^{-1}U\in u$. 
\end{proof}

    



The following result characterizes Schur ultrafilters.

\begin{proposition}\label{char}
For any ultrafilter $u$ on a group $G$ the following conditions are equivalent:
\begin{enumerate}[\rm (i)]
\item $u$ is a Schur ultrafilter;
\item for each $U\in u$ the set $X_U=\{x\in U: \exists y\in U \text{ such that }xy\in U\}$  belongs to $u$;
\item
for each $U\in u$ the set $Y_U=\{y\in U: \exists x\in U \text{ such that }xy\in U\}$  belongs to $u$;
\item $UV\in u$ for any $U,V\in u$;
\item $UU^{-1}\in u$ for each $U\in u$;
\item $U^{-1}U\in u$ for each $U\in u$.
\end{enumerate}
\end{proposition}

\begin{proof}
The implications (ii) $\Rightarrow$ (i) and (iii) $\Rightarrow$ (i) are obvious. 

(i) $\Rightarrow$ (ii): Seeking a contradiction, assume that there exists $U\in u$ such that $X_U\notin u$. Then $W=U\cap (G\setminus X_U)\in u$. Since the ultrafilter $u$ is Schur, there exist $x,y\in W$ such that $xy\in W\subseteq U$. But then $x\in X_U$, which contradicts the definition of $W$. 

The implication (i) $\Rightarrow$ (iii) can be checked similarly.

(i) $\Rightarrow$ (iv): Fix any $U,V\in u$. In order to show that $UV\in u$ fix an arbitrary $W\in u$. Since the ultrafilter $u$ is Schur, there exist $x,y\in U\cap V\cap W$ such that $xy\in U\cap V\cap W$. Taking into account that $xy\in UV\cap W$ and $W$ is arbitrarily chosen, we get that $UV\cap W\neq \emptyset$ for all $W\in u$. Since $u$ is an ultrafilter, we obtain $UV\in u$.

(iv) $\Rightarrow$ (i): Fix any $U\in u$. Since $UU\in u$, the set $UU\cap U$ is not empty. Thus, there exist $x,y\in U$ such that $xy\in U$, witnessing that the ultrafilter $u$ is Schur.

The implication (i) $\Rightarrow$ (v) follows from Lemma~\ref{newrev}(i) (set $v=1_G$). Similarly, the implication (i) $\Rightarrow$ (vi) follows from Lemma~\ref{newrev}(ii).

(v) $\Rightarrow$ (i): Fix any $U\in u$. Since the set $UU^{-1}\cap U$ is not empty, there exist $x,y,z\in U$ such that $xy^{-1}=z$. Then $zy=x\in U$, witnessing that the ultrafilter $u$ is Schur.

(vi) $\Rightarrow$ (i): Fix any $U\in u$. Since the set $U^{-1}U\cap U$ is not empty, there exist $x,y,z\in U$ such that $x^{-1}y=z$. Then $xz=y\in U$, witnessing that the ultrafilter $u$ is Schur.
\end{proof}

For an ultrafilter $u$ on a group $G$, let $u^{-1}=\{U^{-1}:U\in u\}$. Clearly, $u^{-1}$ is also an ultrafilter and the map $u\mapsto u^{-1}$ is a homeomorphism of $\beta G$.
The proof of the following lemma is trivial.
\begin{lemma}
If $u$ is a Schur ultrafilter on a group $G$, then so is $u^{-1}$.    
\end{lemma}


Protasov~\cite[Lemma 5.2]{P} showed that $uu^{-1}$ is a Schur ultrafilter for any $u\in\beta G$. The proof of the following results is based on Protasov's ideas.

\begin{lemma}\label{newProt}
Let $G$ be a group and $u\in\beta G\setminus G$. Then $uu^{-1}$ and $u^{-1}u$ are infinitary Schur.     
\end{lemma}

\begin{proof}
In order to show that $uu^{-1}$ is infinitary Schur, fix any $W\in uu^{-1}$. There exist $U\in u$ and $\{V_x:x\in U\}\subset u$ such that $\bigcup_{x\in U}xV_x^{-1}\subseteq W$. Without loss of generality we can assume that $V_x\subseteq U$ for each $x\in U$. Fix any elements $a\in U$ and $b\in V_a$. Put $V=V_a\cap V_b$. Observe that $ab^{-1}\in  W$, $bV^{-1}$ is an infinite subset of $W$ and $ab^{-1}(bV^{-1})=aV^{-1}\subseteq aV_a^{-1}\subseteq W$. Thus the ultrafilter $uu^{-1}$ is infinitary Schur. 

Note that $u^{-1}\in\beta G\setminus G$ and $u^{-1}u=u^{-1}(u^{-1})^{-1}$. Thus the arguments above imply that $u^{-1}u$ is infinitary Schur.    
\end{proof}

The next proposition generalizes Lemma~\ref{newProt}.

\begin{proposition}\label{Prot}
Let $G$ be a group and $e$ be an idempotent ultrafilter on $G$. Then for any $u\in\beta G\setminus G$ the ultrafilters $ueu^{-1}$ and $u^{-1}eu$ are infinitary Schur.    
\end{proposition}

\begin{proof} 
If $e=1_G$, then $ueu^{-1}=uu^{-1}$ and $u^{-1}eu=u^{-1}u$. By Lemma~\ref{newProt}, the ultrafilters $uu^{-1}$ and $u^{-1}u$ are infinitary Schur. Assume that $e$ is a free idempotent ultrafilter.  
Fix any $F\in ueu^{-1}$. There exist $U\in u$, $\{E_x:x\in U\}\subset e$ and $\{V_h:h\in H=\bigcup_{x\in U}xE_x\}\subset u$ such that $\bigcup_{h\in H}hV_h^{-1}\subseteq F$. Without loss of generality we can assume that $V_h\subset U$ for all $h\in H$. Fix any $x\in U$. Since $e$ is an idempotent ultrafilter, there exist $b\in E_x$ and $I_b\in e$ such that $I_b\subseteq E_x$ as well as $bI_b\subseteq E_x$. Note that $xb\in xE_x\subseteq H$. Fix any $v\in V_{xb}$. It is easy to see that $xbv^{-1}\in F$. Note that the set $C=I_b\cap E_v$ belongs to $e$ and thus is infinite. Since $v\in V_{xb}\subseteq U$ and $C\subseteq E_v$, we get $vc\in H$ for all $c\in C$. Since $C\subseteq I_b$, we get that $bC\subseteq E_x$ and, consequently, $xbC\subseteq H$.
Fix any countably infinite subset $\{c_n:n\in\w\}\subseteq C$. Assume that for each $k\leq n$ we chose points $w_{c_k}\in V_{xbc_k}\cap V_{vc_k}$ such that  $vc_iw_{c_i}^{-1}\neq vc_jw_{c_{j}}^{-1}$ for any distinct $i,j\leq k$. Since $G$ is a group, the set $$X=\{x\in G: vc_{n+1}x^{-1}\in \{vc_iw_{c_i}^{-1}:i\leq n\}\}$$ is finite. Hence we can choose $w_{c_{n+1}}\in V_{xbc_{n+1}}\cap V_{vc_{n+1}}\cap (G\setminus X)$. After completing the induction for each $n\in\w$ we obtain an element $w_{c_n}\in  V_{xbc_{n}}\cap V_{vc_{n}}$ such that the set $F'=\{vc_nw_{c_n}^{-1}:n\in\w\}$ is infinite. 
Since for each $n\in \w$, $vc_n\in H$ and $w_{c_n}\in V_{vc_{n}}$, we get that $F'\subseteq F$. Taking into account that $xbv^{-1}\in F$, $xbc_n\in H$ and $w_{c_n}\in V_{xbc_n}$, we deduce that for every $n\in\w$
$$(xbv^{-1})vc_nw_{c_n}^{-1}=(xbc_n)w_{c_n}^{-1}\in \bigcup_{h\in H} hV_{h}^{-1}\subseteq F.$$
Hence $xbv^{-1}F'\subseteq F$, witnessing that the ultrafilter $ueu^{-1}$ is infinitary Schur. 

Observe that $u^{-1}\in\beta G\setminus G$ and $u^{-1}eu= u^{-1}e(u^{-1})^{-1}$. So the arguments above imply that the ultrafilter $u^{-1}eu$ is infinitary Schur.
\end{proof}

The following result was proven in~\cite[Lemma 5.1]{P}. 
\begin{proposition}[Protasov]\label{pro}
If $G$ is a commutative group, then $\Sch(G)$ is a subsemigroup of $\beta G$.    
\end{proposition}

By $\Sch^{\infty}(G)$ we denote the set of all infinitary Schur ultrafilters on a group $G$.

\begin{proposition}\label{sem}
If $G$ is a commutative group, then $\Sch^{\infty}(G)$ is a left ideal in $\Sch(G)$. In particular, $\Sch^{\infty}(G)$ is a subsemigroup of $\beta G$.   
\end{proposition}

\begin{proof}
Consider any $u\in \Sch(G)$ and $v\in \Sch^{\infty}(G)$. In order to show that $uv\in \Sch^{\infty}(G)$ fix any $F\in uv$. There exist $U\in u$ and $\{V_x:x\in U\}\subset v$ such that $\bigcup_{x\in U}xV_x\subseteq F$. Since the ultrafilter $u$ is Schur there exist $a,b\in U$ such that $ab\in U$. Put $V= V_a\cap V_b\cap V_{ab}$.  Since the ultrafilter $v$ is infinitary Schur, there exist $x\in V$ and an infinite subset $Y\subseteq V$ such that $xY\subseteq V$. Note that $ax\in aV\subseteq aV_a\subseteq F$, and the set $bY\subseteq bV\subseteq bV_b\subseteq F$ is infinite. For each $y\in Y$, 
$$axby=abxy\in ab(xY)\subseteq abV\subseteq abV_{ab}\subseteq F.$$
Hence $ax(bY)\subseteq F$, witnessing that $uv$ is an infinitary Schur ultrafilter.
\end{proof}

Next we turn to combinatorial properties of Schur ultrafilters. For this we need the following result proved in~\cite{Katetov}.

\begin{theorem}[Kat\v{e}tov]\label{Kat}
Let $f:X\rightarrow X$ be a map such that $f(x)\neq x$ for all $x\in X$. Then there exist disjoint subsets $A_0, A_1, A_2$ of $X$ such that $A_0\cup A_1\cup A_2=X$ and $f(A_i)\cap A_i=\emptyset$ for all $i\in\{0,1,2\}$.   
\end{theorem}

\begin{proposition}\label{finschur}
Let $u$ be a free Schur ultrafilter on a group $G$. Then for every $U\in u$ and $n\in\w$ there exists $x\in U$ such that $|\{y\in U: xy\in U\}|>n$ or $|\{y\in U: yx\in U\}|> n$.    
\end{proposition}

\begin{proof}
Seeking a contradiction, assume that there exist a free Schur ultrafilter $u$ on $G$, $U\in u$ and $n\in\w$ such that for any $x\in U$ the sets $L_x=\{y\in U:xy\in U\}$ and $R_x=\{y\in U: yx\in U\}$ contain at most $n$ elements.

\begin{claim}\label{W}
There exists $W\in u$ such that $W\subseteq U$ and $|W\cap L_x|\leq 1$ for all $x\in U$.   \end{claim}
\begin{proof}
By Zorn's Lemma, there exists a maximal with respect to the inclusion subset $W_1$ of $U$ such that $|W_1\cap L_x|\leq 1$ for all $x\in U$. If $W_1\in u$ we are done. Otherwise, $U\setminus W_1\in u$. Using Zorn's Lemma once again, find maximal $W_2\subset U\setminus W_1$ such that $|W_2\cap L_x|\leq 1$ for all $x\in U$. Either $W_2\in u$ and we are done, or $U\setminus (W_1\cup W_2)\in u$. 
Proceeding this way we either find $i\leq (n+1)^2$ such that $W_i\in u$ and $|W_i\cap L_x|\leq 1$ for all $x\in U$, or obtain a family $\{W_i:i\leq (n+1)^2\}$ of pairwise disjoint subsets of $U$ such that for every $i\leq (n+1)^2$, $W_i\notin u$ and $W_i$ is maximal with respect to the inclusion subset of $U\setminus (W_1\cup\ldots \cup W_{i-1})$ such that $|W_i\cap L_x|\leq 1$ for all $x\in U$. Let us show that the latter variant is impossible. Indeed, in this case  the set $U\setminus (\bigcup_{i\leq (n+1)^2} W_i)\in u$. Hence there exist $x\in U$ and $y\in L_x$ such that $y\notin \bigcup_{i\leq (n+1)^2}W_i$. Let $N=\{i\leq (n+1)^2: W_i\cap L_x\neq \emptyset\}$. Since the sets $W_i$, $i\leq (n+1)^2$ are pairwise disjoint, the inequality $|N|>n$ would yield that $|L_x|> n$, which contradicts the assumption. Thus $|N|\leq n$. It follows that the set $M=\{i\leq (n+1)^2: W_i\cap L_x= \emptyset\}$ contains at least $n^2+n+1$ elements. For each $i\in M$ the maximality of $W_i$ implies that there exists $z(i)\in U$ such that $y\in L_{z(i)}$ and $W_i\cap L_{z(i)}\neq \emptyset$. Since $z(i)\in R_y$ for any $i\in M$, and $|R_y|\leq n$, the pigeonhole principle yields a subset $J\subset M$ such that $|J|>n$ and $z(i)=z(j)$ for any $i,j\in J$. Fix any $j\in J$. Since $W_i\cap L_{z(j)}\neq \emptyset$ for all $i\in J$, and $W_n\cap W_m=\emptyset$ whenever $n\neq m$, we get that $|L_{z(j)}|\geq |J|>n$ which contradicts the assumption.
\end{proof}

By Claim~\ref{W}, we can pick $W\in u$ such that for each $x\in W$ there exists at most one $y\in W$ satisfying $xy\in U$.
By Proposition~\ref{char}(ii), 
$W'=\{x\in W: \exists y_x\in W \hbox{ such that } xy_x\in W\}\in u$.

There are two cases to consider:
\begin{enumerate}
    \item $T=\{x\in W':y_x\neq x\}\in u$;
    \item $S=\{x\in W':y_x= x\}\in u$.
\end{enumerate}

(1) Fix any $z\in T$ and define a map $f: W\rightarrow W$ as follows: 
$$
f(x)=\begin{cases} 
y_x, &\hbox{if } x\in T;\\
z, &\hbox{otherwise}.
\end{cases}
$$ 
It is clear that $f(x)\neq x$ for all $x\in W$. By Theorem~\ref{Kat}, there exist disjoint subsets $A, B, C$ of $W$ such that $A\cup B\cup C=W$ and $f(A)\cap A=f(B)\cap B= f(C)\cap C=\emptyset$. Without loss of generality we can assume that $A\in u$. Let $A'=A\cap T$. By the assumption, $A'\in u$. Proposition~\ref{char}(iii) and the definition of $f$ imply that $f(A')\in u$. But then $\emptyset=A'\cap f(A')\in u$, which implies a contradiction.

(2) Since $S\subseteq W'$, we get $x^2\in W$ for all $x\in S$ and $xy\notin W$ for any distinct $x,y\in S$, i.e. $SS\cap W=\{x^2:x\in S\}$. Shrinking $S$ if necessary, we can additionally assume that $1_G\notin S$, implying that $x\neq x^2$ for all $x\in S$.
Fix any $s\in S$.
Consider the map $g: W\rightarrow W$ defined by $$g(x)=
\begin{cases}
x^2, &\hbox{if } x\in S;\\
s, &\hbox{otherwise}.
\end{cases}
$$
Since $1_G\notin S$, we get that $g(x)\neq x$ for all $x\in W$. Theorem~\ref{Kat} yields pairwise disjoint subsets $A,B,C$ of $W$ such that $A\cup B\cup C=W$ and $f(A)\cap A=f(B)\cap B= f(C)\cap C=\emptyset$. Without~loss of generality we can assume that $A\in u$. Then $A'=A\cap S\in u$.  
By the choice of $S$, we have that $$(A'A')\cap W=\{x^2:x\in A'\}=g(A').$$ Since $u$ is a Schur ultrafilter, Proposition~\ref{char}(iv) implies that $A'A'\in u$. It follows that $g(A')\in u$ and, consequently, $\emptyset=A'\cap g(A')\in u$, which is a desired contradiction. 
\end{proof}

Proposition~\ref{finschur} implies the following corollary.

\begin{corollary}\label{nice}
Let $u$ be a Schur ultrafilter on a commutative group $G$. Then for each $U\in u$ and $n\in\w$ there exists $x\in U$ such that $|\{y\in U:xy\in U\}|> n$.    
\end{corollary}

Next we are going to show that infinitary Schur ultrafilters and Schur ultrafilters are in some sense antipodes to P-points and selective ultrafilters, respectively.  
Recall that an ultrafilter $u$ on a countable set $X$ is called 
\begin{enumerate}[\rm (i)]
    \item {\em weak P-point} if $u$ is not in the closure of any countable subset of $\beta X\setminus (X\cup \{u\})$;
    \item {\em P-point} if for every partition $\{C_n:n\in\w\}$ of $X$ such that $C_n\notin u$ for all $n\in\w$, there exists $U\in u$ such that the set $U\cap C_n$ is finite for every $n\in\w$; 
    \item {\em selective} if for every partition $\{C_n:n\in\w\}$ of $X$ such that $C_n\notin u$ for all $n\in\w$, there exists $U\in u$ such that $|U\cap C_n|\leq 1$ for all $n\in\w$. 
\end{enumerate}
It is known that each selective ultrafilter is a P-point and each P-point is a weak P-point.
The celebrated ZFC result of Kunen~\cite{Kunen} states that $\beta \w\setminus \w$ contains a weak P-point. However, the existence of P-points and selective ultrafilters in $\beta \w\setminus \w$ is consistent with ZFC, but cannot be derived from ZFC alone (see~\cite{Blass, Chodounsky, Shelah}). It is well known that an idempotent ultrafilter $u$ on a countable group $G$ cannot be a weak P-point, as $u$ is an accumulation point of the family $\{xu: x\in G\}\subset \beta G$.

\begin{proposition}\label{consistent}
Let $G$ be a countable group and $u\in\beta G$ be an infinitary Schur ultrafilter. Then $u$ is not a P-point.
\end{proposition}

\begin{proof}
 Seeking a contradiction, assume that there exists a P-point $u\in \Sch^{\infty}(G)$. Since $u$ is not an idempotent,  there exists $U\in u$ such that for each $x\in U$ the set $L_x=\{y\in U: xy\in U\}$ does not belong to $u$. By Proposition~\ref{char}(iii), $Y_U=\bigcup_{x\in U}L_x\in u$. Let $U=\{x_n:n\in\w\}$. Put $C_0=L_{x_0}$ and for each $n>0$ let $C_n=L_{x_n}\setminus \bigcup_{i<n}L_{x_i}$. Clearly, the family $\{C_n:n\in\w\}$ is a partition of $U$ (we allow here that some $C_n$ can be empty). Since $u$ is a P-point, there exists $W\in u$ such that the set $W\cap C_n$ is finite for each $n\in\w$. Since $u$ is an infinitary Schur ultrafilter, there exist $x\in W$ and an infinite subset $Z\subseteq W\cap L_x$ such that $xZ\subseteq W$. On the other hand, for any $n\in\w$ the set $W\cap L_{x_n}$ is finite, as it is contained in $\bigcup_{i\leq n}C_i$. The obtained contradiction completes the proof.
\end{proof}

\begin{proposition}\label{select}
 Let $u$ be a free Schur ultrafilter on a countable commutative group $G$. Then $u$ is not selective.   
\end{proposition}
\begin{proof}
 Seeking a contradiction, assume that there exists a free selective ultrafilter $u\in \Sch(G)$.
 By Proposition~\ref{consistent}, $u$ is not infinitary Schur.  Then there exists $U\in u$ such that for each $x\in U$ the set $L_x=\{y\in U: xy\in U\}$ is finite.

 \begin{claim}\label{finite}
  For each $x\in U$ the set $Z=\{z\in U: L_z\cap L_x\neq \emptyset\}$ is finite.   
 \end{claim}

 \begin{proof}
Assume to the contrary that there exists $x\in U$ such that $L_x\cap L_z\neq \emptyset$ for infinitely many $z\in U$. Since the set $L_x$ is finite, the pigeonhole principle yields $y\in L_x$ and an infinite subset $Z'\subseteq Z$ such that $y\in L_z$ for all $z\in Z'$. Then for any $z\in Z'$, we get that $yz=zy\in U$. Hence $L_y$ contains an infinite subset $Z'$, which contradicts our assumption.
 \end{proof}

Fix an enumeration $\{x_n:n\in\w\}$ of $U$. 
We shall inductively construct an auxiliary partition of $U$ into finite sets $\{B_n: n\in\w\}$. We begin with $B_0=\{x_0\}$. Assume that finite pairwise disjoint subsets $B_i\subset U$ are already constructed for all $i\leq n$. Let $j=\min \{m\in\w: x_m\notin\bigcup_{i\leq n}B_i\}$. Put 
 $$B_{n+1}=\{x\in U\setminus \bigcup_{i\leq n}B_i: L_x\cap L_z\neq \emptyset \hbox{ for some }z\in B_n\}\cup \{x_j\}.$$
 Claim~\ref{finite} implies that each $B_n$ is finite.
For each $n\in\w$, let $C_n=\bigcup_{x\in B_n}L_{x}$. It is easy to see that $C_n\cap C_{m}=\emptyset$ whenever $|n-m|\geq 2$.  
By Proposition~\ref{char}(iii), $Y_U=\bigcup_{x\in U}L_x=\bigcup_{n\in\w}C_n\in u$. Since $u$ is an ultrafilter, either $\bigcup_{n\in\w}C_{2n}\in u$ or $\bigcup_{n\in\w}C_{2n+1}\in u$. Without loss of generality, assume that $W=\bigcup_{n\in\w}C_{2n}\in u$. 
Since $\{C_{2n}:n\in\w\}$ is a partition of $W$ consisting of finite sets and the ultrafilter $u$ is selective, there exists $F\in u$ such that $|F\cap C_{2n}|\leq 1$ for every $n\in \w$. Fix any $x\in F\cap U$. There exists $m\in \w$ such that $L_x\subseteq C_m$. Since $C_n\cap C_{m}=\emptyset$ whenever $|n-m|\geq 2$, we obtain that $L_x\cap C_n=\emptyset$ for all $n\in \w\setminus \{m-1,m,m+1\}$. By the choice of $F$, $|\{y\in U\cap F: xy\in U\cap F\}|\leq 3$. Thus for each $x\in F\cap U$ we have $|\{y\in U\cap F: xy\in U\cap F\}|\leq 3$, which contradicts Corollary~\ref{nice}.
\end{proof}

The following two problems are left open.

\begin{problem}\label{prob1}
Does there exist a free Schur ultrafilter on a (countable) group $G$ that is not infinitary Schur?    
\end{problem}

\begin{problem}\label{prob2}
Is the existence of a free Schur ultrafilter on $\mathbb Z$ that is a P-point consistent with ZFC?    
\end{problem}

Observe that the positive solution of Problem~\ref{prob2} implies the positive consistent solution of Problem~\ref{prob1}. Noteworthy, there exists a model of ZFC where $\beta\omega\setminus \omega$ contains P-points but all of them are selective (see~\cite[Section XVIII.4]{Shelahbook}). Taking into account Proposition~\ref{select}, we obtain that the positive answer to Problem~\ref{prob2} cannot be derived from the existence of free P-points alone. Also, if $G$ is a countable group, then no P-point $u\in \beta G\setminus G$ can be represented as $u=vw$ for some $v,w\in\beta G\setminus G$, because $u$ would be an accumulation point of the countable subset $\{gw: g\in G\}\subset \beta G\setminus G$. 

Recall that a semigroup $S$ endowed with a topology $\tau$ is called a {\em topological semigroup} if the multiplication is continuous in $(S,\tau)$.
The following proposition presents another similarity between Schur and idempotent ultrafilters.

\begin{proposition}\label{Schur}
Let $G$ be a group, $S$ be a topological semigroup and $f:\beta G\rightarrow S$ be a continuous homomorphism. Then for every Schur ultrafilter $u\in\beta G$ the following assertions hold:
\begin{enumerate}[\rm (i)]
    \item $f(u)$ is an idempotent;
    \item there exists an idempotent ultrafilter $e\in\beta G$ such that $f(e)=f(u)$.
\end{enumerate}
\end{proposition}

\begin{proof}
(i) By the continuity of $f$, the filter $\F$ on $S$ generated by the family $\{f(U):U\in u\}$ converges to $f(u)$. Fix an arbitrary open neighborhood $W$ of $f(u)f(u)$. Since $S$ is a topological semigroup and the map $f$ is continuous, there exists $U\in u$ such that $f(UU)=f(U)f(U)\subseteq W$. By Proposition~\ref{char}(iv), the set $UU$ belongs to $u$ and, consequently, $f(UU)\in \F$. Since the neighborhood $W$ was chosen arbitrarily, we get that $\F$ converges to $f(u)f(u)$. Taking into account that in a Hausdorff space each filter converges to at most one point, we get that $f(u)f(u)=f(u)$. 

(ii) By item (i), the element $y=f(u)\in S$ is an idempotent. Then $f^{-1}(y)$ is a compact subsemigroup of $\beta G$. Theorem~\ref{classic} implies that $f^{-1}(y)$ contains an idempotent $e$. It is clear that $f(e)=f(u)$.
\end{proof}

We finish this section with the following two propositions describing topological properties of the spaces $\Sch(G)$ and $\Sch^{\infty}(G)$.

\begin{lemma}\label{closed}
For any group $G$ the sets $\Sch(G)$ and $\Sch^{\infty}(G)$ are closed in $\beta G$.  
\end{lemma}

\begin{proof}   
Fix any $u\in\beta G\setminus \Sch^{\infty}(G)$ and $U\in u$ such that for each $x\in U$ the set $\{y\in U:xy\in U\}$ is finite. Then the basic open neighborhood $\langle U\rangle$ of $u$ is disjoint with $\Sch^{\infty}(G)$. Closedness of the set $\Sch(G)$ can be proved analogously. 
\end{proof}

\begin{proposition}\label{nowheredense}
For any group $G$ the set of all free Schur ultrafilters is nowhere dense in $\beta G\setminus G$.  
\end{proposition}

\begin{proof}
Lemma~\ref{closed} implies that the set $\Sch(G)\setminus G$ is closed in $\beta G\setminus G$. Seeking a contradiction, assume that the interior of $\Sch(G)\setminus G$ in $\beta G\setminus G$ is not empty. Then there exists an infinite subset $A\subseteq G$ such that $\langle A\rangle\setminus G\subseteq \Sch(G)$. Shrinking $A$ if necessary, we can assume that $1_G\notin A$. We are going to construct by induction an infinite subset $B\subseteq A$ such that $BB\cap B=\emptyset$. 
 Put 
$A^{[2]}=\{x^2:x\in A\}$.
The following two cases are possible: 
\begin{itemize}
    \item[(i)] the set $A\setminus A^{[2]}$ is finite;
    \item [(ii)] the set $A\setminus A^{[2]}$ is infinite.
\end{itemize}
(i) Fix any $b_0\in A$ and let $B_0=\{b_0\}$. Since $b_0\neq 1_G$, we have that $B_0\cap B_0B_0=\emptyset$. Assume that we already constructed a subset $B_n=\{b_i:i\leq n\}\subset A$ such that $B_nB_n\cap B_n=\emptyset$.   Consider the set $$C_n=\{x\in A: \exists i,j\leq n \text{ such that } xb_i=b_j \text{ or }b_ix=b_j\}.$$
Then $C_n\subseteq \{b_jb_i^{-1}: i,j\leq n\}\cup \{b_i^{-1}b_j: i,j\leq n\}$, so $C_n$ is finite. Let $D_n=\{x\in A:x^2\in B_n\}$. If the set $A\setminus D_n$ is finite, then  $A^{[2]}$ is finite, because $$A^{[2]}=\{x^{2}: x\in A\setminus D_n\}\cup \{x^{2}: x\in D_n\}\subseteq \{x^{2}: x\in A\setminus D_n\}\cup B_n.$$
It follows that $A\setminus A^{[2]}$ is infinite, which contradicts our assumption. Hence the set $A\setminus D_n$ is infinite. 
Fix any $b_{n+1}\in A\setminus (B_n\cup B_nB_n\cup C_n\cup D_n)$. It is easy to check that after completing the induction we obtain an infinite subset $B=\{b_i:i\in\w\}$ of $A$ such that $BB\cap B=\emptyset$.

(ii) Put $A'=A\setminus A^{[2]}$ and fix any $b_0\in A'$. Assume that we have already constructed a subset $B_n=\{b_i:i\leq n\}\subset A'$ such that $B_nB_n\cap B_n=\emptyset$.
Similarly as in case (i), put $$C_n=\{x\in A: \exists i,j\leq n \text{ such that } xb_i=b_j \text{ or }b_ix=b_j\}.$$
Let $b_{n+1}$ be arbitrary element of $A'\setminus (B_n\cup B_nB_n\cup C_n)$. After completing the induction we obtain an infinite set $B=\{b_i:i\in\w\}$ of $A$ such that $BB\cap B=\emptyset$.

Note that for each free ultrafilter $u\in\langle B\rangle\subseteq \langle A\rangle$ we have $u\notin \Sch(G)$, which contradicts our assumption. Thus, the interior of $\Sch(G)\setminus G$ is empty, as required. 
\end{proof}

\section{Bohr compactification of topological groups}\label{sec3} 

The following proposition follows from~\cite[Theorem 4.8]{HS}.
\begin{proposition}\label{key}
Let $X$ be a discrete semigroup, $Y$ a compact right topological semigroup, and $f:X\rightarrow Y$ a homomorphism such that $f(X)\subseteq \Lambda(Y)$. Then $f$ can be extended to a unique continuous homomorphism $\phi: \beta X\rightarrow Y$.       
\end{proposition}

We shall refer several times to the following folklore fact, which is a corollary of~\cite[Lemma 1.47]{CHK}.

\begin{lemma}[Folklore]\label{lem}
Let $X$ and $Z$ be topological spaces, $\rho$ an equivalence relation on $X$, and $f: X\rightarrow Z$, $g: X/\rho\rightarrow Z$ be maps such that $f=g\circ h_\rho$. Then $g$ is continuous if and only if $f$ is continuous.
\end{lemma}

Recall that for a given topological group $G$, by $\Theta$  we denote the least closed congruence on $\beta G_d$ that contains the relation $\{(u,1_G):u \hbox{ is idempotent}\}$ (see Definition~\ref{def}).

\begin{proposition}\label{old}
Let $G$ be a discrete group and $\rho$ a closed congruence on $\beta G$. Then $\Theta\subseteq \rho$  if and only if  the quotient semigroup $\beta G/\rho$ is a chart group satisfying $h_\rho(G)\subseteq \Lambda(\beta G/\rho)$.   
\end{proposition}

\begin{proof}
 ($\Rightarrow$): Let $\rho$ be a closed congruence on $\beta G$ such that $(e,1_G)\in \rho$ for each idempotent ultrafilter $e$. 
Hence for each idempotent ultrafilter $e\in \beta G$ and $u\in \beta G$ we have $(eu,1_G u)=(eu,u)\in \rho$. \cite[Theorem~1]{Zlatos} implies that $\beta G/\rho$ is a compact right topological group. Clearly, $h_\rho(G)$ is a dense subgroup of $\beta G/\rho$. It remains to check that $h_\rho(G)\subseteq \Lambda(\beta G/\rho)$. Note that for each $u\in \beta G$, 
\begin{equation}\label{eqq}
h_\rho(\lambda_u(x))=[ux]_\rho=[u]_\rho[x]_\rho=\lambda_{[u]}(h_\rho(x)).
\end{equation}
Recall that for each $g\in G$ the left shift $\lambda_g$ is continuous in $\beta G$, implying that the map $h_\rho\circ \lambda_g$ is continuous too. Equation~\ref{eqq} follows that $h_\rho\circ \lambda_u=\lambda_{[u]}\circ h_\rho$ for every $u\in \beta G$. Thus the map $\lambda_{[g]}\circ h_\rho$ is continuous for each $g\in G$. By Lemma~\ref{lem}, the left shift $\lambda_{[g]}:\beta G/\rho\rightarrow \beta G/\rho$ is continuous for each $g\in G$. Hence $h_\rho(G)\subseteq \Lambda(\beta G/\rho)$. 

($\Leftarrow$): Since a homomorphic image of an idempotent is an idempotent and $\rho$ is a group congruence, we get $(e,1_G)\in\rho$ for each idempotent $e\in\beta G$. Hence $\Theta \subseteq \rho$.
\end{proof}

Recall that for a given topological group $G$, by $\Phi$ we denote the least closed congruence on $\beta G_d$ that contains the relation $\{(u,1_G):u \hbox{ is idempotent}\}\cup \{(u,1_G): u\in\Ult(G)\}$ (see Definition~\ref{def}). Also recall that for a given congruence $\rho$ on $\beta G_d$ the map $\mathfrak h_\rho: G\rightarrow \beta G_d/\rho$ is defined by $\mathfrak h_\rho(g)=h_\rho(g)=[g]_\rho$ for all $g\in G$ (see Definition~\ref{map}).

\begin{proposition}\label{gennew}
Let $G$ be a right topological group, and $\rho$ a closed group congruence on $\beta G_d$. Then the map $\mathfrak h_\rho:  G\rightarrow \beta G_d/\rho$ is continuous if and only if $\Phi\subseteq \rho$.    
\end{proposition}

\begin{proof}
($\Rightarrow$): Since $\rho$ is a group congruence, $(e,1_G)\in\rho$ for each idempotent $e\in\beta G_d$. Fix any  ultrafilter $u\in \Ult(G)$, i.e., $u$ converges to $1_G$ in $G$. By the continuity of the map $\mathfrak h_\rho$, the ultrafilter $\mathfrak h_\rho(u)$ on $\beta G_d/\rho$ generated by the family $\{\mathfrak h_\rho(U):U\in u\}$ converges to $\mathfrak h_\rho(1_G)=h_\rho(1_G)$. Since the map $h_\rho$ is continuous, for each open neighborhood $W\subseteq \beta G_d/\rho$ of $h_\rho(u)$ there exists $U\in u$ such that $h_\rho(U)=\mathfrak h_\rho(U)\subseteq W$. Hence the ultrafilter $\mathfrak h_\rho (u)$ converges to $h_\rho(u)$. The closedness of the congruence $\rho$ implies that the space $\beta G_d/\rho$ is Hausdorff. It follows that $h_\rho(u)=h_\rho(1_G)$ and, consequently, $(u,1_G)\in\rho$. Hence $\Phi\subseteq \rho$.

($\Leftarrow$):
By Proposition~\ref{old}, $\beta G_d/\rho$ is a chart group. To establish the continuity of the map $\mathfrak h_\rho$, it suffices to check that for any $u\in \Ult(G)$ the ultrafilter $\mathfrak h_\rho(u)$ on $\beta G_d/\rho$ generated by the family $\{\mathfrak h_\rho(U): U\in u\}$ converges to $[1_G]_{\rho}$. 
For this fix arbitrary $u\in \Ult(G)$ and open neighborhood $W$ of $[1_G]_\rho$ in $\beta G_d/\rho$. Since $\Phi\subseteq \rho$ we get that $(u,1_G)\in\rho$. It follows that $h_\rho^{-1}(W)$ is an open neighborhood of $u$ in $\beta G_d$. Then $V=h_\rho^{-1}(W)\cap G\in u$. Observe that 
$$\mathfrak h_\rho(V)=h_\rho(V)\subseteq h_\rho(h_\rho^{-1}(W)) \subseteq W,$$ witnessing that $W\in \mathfrak h_\rho(u)$. Hence the ultrafilter $\mathfrak h_\rho(u)$ converges to $[1_G]_\rho$.
\end{proof}

The following renowned theorem was proven in~\cite{E}.

\begin{theorem}[Ellis]\label{classic1}
A locally compact semitopological group is a topological group.    
\end{theorem}

Recall that for each ultrafilter $u$ on a group $G$ by $u^{-1}$ we denote the ultrafilter $\{U^{-1}:U\in u\}$, and for any homomorphism $f:X\rightarrow Y$, $\Ker(f)=\{(x,y)\in X{\times}X: f(x)=f(y)\}$ is a congruence on $X$. If, moreover, $Y$ is a Hausdorff topological semigroup, then it is easy to check that the congruence $\Ker(f)$ is closed.
We are in a position to prove Theorem~\ref{topgr}, i.e., to show that for any topological group $G$ the pair $(\beta G_d/\Psi,\mathfrak h_\Psi)$ is the Bohr compactification of $G$.

\begin{proof}[{\bf Proof of Theorem~\ref{topgr}}]
By Proposition~\ref{old}, $\beta G_d/\Psi$ is a chart group. Since for each ultrafilter $u$ on $G$, $uu^{-1}$ is a Schur ultrafilter (see Proposition~\ref{Prot}), and $\{(u,1_G): u\in \Sch(G)\}\subset \Psi$, we get that  
$$[u]_{\Psi}[u^{-1}]_{\Psi}=[uu^{-1}]_{\Psi}=[1_G]_{\Psi}.$$ Thus, $[u^{-1}]_{\Psi}$ is the inverse element of $[u]_{\Psi}$ in the group $\beta G_d/\Psi$. 
Let $\operatorname{inv}:\beta G_d/\Psi\rightarrow \beta G_d/\Psi$ be the inversion.
Define a map $\iota:\beta G_d\rightarrow \beta G_d$ by $\iota(u)=u^{-1}$. The arguments above imply that for any $u\in\beta G_d$, $$h_\Psi(\iota(u))=\operatorname{inv}(h_\Psi(u)).$$  
Since $\iota$ is a homeomorphism we obtain that the map $h_\Psi\circ \iota$ is continuous. By Lemma~\ref{lem}, the map $\operatorname{inv}: \beta G_d/\Psi\rightarrow \beta G_d/\Psi$ is continuous. Observe that for each $p\in \beta G/\Psi$ we have the following:
$$\lambda_{p}=\operatorname{inv}\circ \rho_{\operatorname{inv}(p)}\circ \operatorname{inv}.$$
Taking into account that $\beta G_d/\Psi$ is a right topological group, the formula above implies that left shifts are continuous in $\beta G_d/\Psi$ and, as such, $\beta G_d/\Psi$ is a compact semitopological group. By Theorem~\ref{classic1}, $\beta G_d/\Psi$ is a compact topological group.
 Since $\Phi\subseteq \Psi$, Proposition~\ref{gennew} implies that the map $\mathfrak h_\Psi: G\rightarrow \beta G_d/\Psi$ is continuous.

In order to show that $\beta G_d/\Psi$ is the Bohr compactification of $G$, fix any continuous homomorphism $f$ from $G$ to a compact topological group $H$. Without loss of generality we can assume that $f(G)$ is dense in $H$.
Proposition~\ref{key} yields a continuous homomorphism $\phi:\beta G_d\rightarrow H$, such that $\phi(x)=f(x)$ for each $x\in G$. Since $f(G)$ is dense in $H$, the map $\phi$ is surjective. Observe that the congruence $\Ker(\phi)$ on $\beta G_d$ is closed, as $\phi$ is continuous and the space $H$ is Hausdorff.
By Proposition~\ref{Schur}, for each Schur ultrafilter $u$ on $G$ the element  $\phi(u)$ is an idempotent, witnessing that $\phi(u)=1_H=\phi(1_G)$. Thus $\{(u,1_G): u\in \Sch(G)\}\subset \Ker(\phi)$.
Consider any ultrafilter $u\in \Ult(G)$. By the continuity of $f$, the ultrafilter $f(u)$ on $H$ generated by the family $\{f(U): U\in u\}$ converges to $f(1_G)=1_H$.
The continuity of $\phi$ implies that for each open neighborhood $V$ of $\phi(u)$ there exists $U\in u$ such that $\phi(U)=f(U)\subseteq V$. It follows that the ultrafilter $f(u)$ converges to $\phi(u)$. Since the space $H$ is Hausdorff, 
$\phi(u)=1_H=\phi(1_G)$ and, consequently, $(u,1_G)\in \Ker(\phi)$.  Hence $\{(u,1_G): u\in\Ult(G)\}\subset \Ker(\phi)$, witnessing that $\Psi\subseteq \Ker(\phi)$. By Theorem~1.48 from~\cite{CHK}, there exists a continuous homomorphism $g: \beta G_d/\Psi\rightarrow H$ such that $\phi=g\circ h_\Psi$. Recall that for each $x\in G$, $\mathfrak h_\Psi(x)=h_\Psi(x)$ and $\phi(x)=f(x)$, witnessing that $f(x)=g(\mathfrak h_\Psi(x))$. Hence the pair $(\beta G_d/\Psi,\mathfrak h_\Psi)$ is the Bohr compactification of $G$. 
\end{proof}

The following proposition characterizes compact topological groups. 

\begin{proposition}\label{autom}
Let $G$ be a group endowed with a topology. Then $G$ is a compact topological group if and only if there exist a discrete group $H$ and a continuous surjective homomorphism $g:\beta H\rightarrow G$ such that $g(u)=1_G$ for any Schur ultrafilter $u$ on $H$.
\end{proposition}

\begin{proof}
($\Rightarrow$) Fix a compact topological group $G$. Let $H$ be the group $G$ endowed with the discrete topology, and $i: H\rightarrow G$ be the identity homomorphism. By Proposition~\ref{key}, there exists a continuous homomorphism $g:\beta H\rightarrow G$ that extends $i$. It is clear that $g$ is surjective. Proposition~\ref{Schur} implies that $g(u)=1_G$ for any Schur ultrafilter $u$ on $H$.

($\Leftarrow$) Assume that $G$ is a group equipped with a Hausdorff topology, and there exists a discrete group $H$ together with a continuous surjective homomorphism $g:\beta H\rightarrow G$ such that $g(u)=1_G$ for any Schur ultrafilter $u$ on $H$. 
The definition of $g$ implies that $\Xi \subseteq \Ker(g)$, where $\Xi$ is the congruence on $\beta H$ from Definition~\ref{def}.

Define a map $f: \beta H/\Xi\rightarrow G$ by $f([x]_\Xi)=g(x)$. Since 
$$f([x]_\Xi[y]_\Xi)=f([xy]_\Xi)=g(xy)= g(x) g(y)=f([x]_\Xi)f([y]_\Xi),$$
we obtain that $f$ is a surjective homomorphism. Observe that $g=f\circ h_\Xi$. Since the map $g$ is continuous, Lemma~\ref{lem} implies that the homomorphism $f$ is continuous. By Theorem~\ref{Bohr}, $\beta H/\Xi$ is the Bohr compactification of $H$. 
Then $G$, being a continuous homomorphic image of $\beta H/\Xi$, is a compact topological group as well.
\end{proof}






\section{Chart groups}\label{sec4}

We are now in a position to prove Theorem~\ref{chartgr}, i.e., to show that for any right topological group $G$ the pair $(\beta G_d/\Phi,\mathfrak h_\Phi)$ is the universal chartification of $G$.

\begin{proof}[{\bf Proof of Theorem~\ref{chartgr}}]
Proposition~\ref{old} implies that $\beta G_d/\Phi$ is a chart group and $\phi(G)\subseteq \Lambda (\beta G_d/\Phi)$. By Proposition~\ref{gennew}, the map $\mathfrak h_\Phi: G\rightarrow \beta G_d/\Phi$ is continuous.

Fix any chart group $H$ and a continuous homomorphism $g:G\rightarrow H$ such that $g(G)\subseteq \Lambda (H)$ is dense in $H$. Then $g$ can be also viewed as a continuous homomorphism from $G_d$ to $H$. Proposition~\ref{key} yields a continuous homomorphic extension $\phi:\beta G_d\rightarrow H$ of $g$. Note that $\phi$ is surjective, as $g(G)$ is dense in $H$.  
Since $\phi$ is continuous and $H$ is a chart group, we get that $\Ker(\phi)$ is a closed group congruence. It follows that $\Theta\subseteq \Ker(\phi)$. 
Consider any ultrafilter $u\in \Ult (G)$. By the continuity of $g$, the ultrafilter $g(u)$ on $H$ generated by the family $\{g(U): U\in u\}$ converges to $g(1_G)=1_H$.
The continuity of $\phi$ implies that for each open neighborhood $V$ of $\phi(u)$ there exists $U\in u$ such that $\phi(U)=g(U)\subseteq V$. It follows that the ultrafilter $g(u)$ converges to $\phi(u)$. Since the space $H$ is Hausdorff, $\phi(u)=1_H=\phi(1_G)$ and, consequently, $(u,1_G)\in \Ker(\phi)$. Thus $\{(u,1_G): u\in\Ult(G)\}\subset \Ker(\phi)$, and hence  $\Phi\subseteq \Ker(\phi)$. 

Define a map $f: \beta G_d/\Phi\rightarrow H$ by $f([x]_\Phi)=\phi(x)$. Since 
$$f([x]_\Phi[y]_\Phi)= f([xy]_\Phi)= \phi(xy)= \phi(x)\phi(y)=f([x]_\Phi)f([y]_\Phi),$$
we obtain that $f$ is a surjective homomorphism. Observe that $\phi=f\circ h_\Phi$. Since the map $\phi$ is continuous, Lemma~\ref{lem} implies that $f$ is continuous. Finally, for each $y\in G$, $$f(\mathfrak h_\Phi(y))=f(h_\Phi(y))=f([y]_\Phi)=\phi(y)=g(y).$$ Hence the pair $(\beta G_d/\Phi,\mathfrak h_\Phi)$ is the universal chartification of $G$. 
\end{proof}

Recall that $\Lambda (G)$ is a subgroup of $G$ for any chart group $G$ (see ~\cite[Proposition 1]{R0}).

\begin{proposition}\label{idempotent}
Let $u$ be an idempotent ultrafilter on a chart group $G$. If $\Lambda(G)\in u$, then $u$ converges to $1_G$. \end{proposition}

\begin{proof}
Since the group $G$ is compact, there exists $x\in G$ such that $u$ converges to $x$. Pick an open neighborhood $W$ of $x^2$. Since $G$ is a chart group, there exists $U\in u$ such that $Ux\subseteq W$. Observe that $V=U\cap \Lambda(G)\in u$. Since $V\subseteq \Lambda (G)$ and $u$ converges to $x$, for each $a\in V$ there exists $V_a\in u$ such that $aV_a\subseteq W$. Thus $\bigcup_{a\in V}aV_a\subseteq W$. As $u$ is an idempotent ultrafilter, we get that $\bigcup_{a\in V}aV_a\in u$. It follows that $W\in u$, witnessing that $u$ converges to $x^2$. Since $G$ is Hausdorff, $x=x^2=1_G$.  
\end{proof}

\begin{proposition}\label{compSchur}
Each Schur ultrafilter on a compact topological group $G$ converges to $1_G$. 
\end{proposition}

\begin{proof}
Let $u$ be a Schur ultrafilter on $G$. 
Since $G$ is compact, there exists $x\in G$ such that $u$ converges to $x$. Fix any open neighborhood $W$ of $x^2$. Since $G$ is a topological group, there exists $U\in u$ such that $UU\subseteq W$. By Proposition~\ref{char}(iv), $UU\in u$, which implies that $u$ converges to $x^2$. Since $G$ is Hausdorff, we get $x=x^2=1_G$.     
\end{proof}

We are in a position to prove Theorem~\ref{main}. Recall that we need to show the equivalence of the following assertions for every chart group $G$.
\begin{enumerate}[\rm (i)]
    \item $G$ is a topological group;
    \item every Schur ultrafilter on $G$ converges to $1_G$;
    \item there exists a dense subgroup $H\subseteq \Lambda(G)$ such that each Schur ultrafilter on $H$ converges to $1_G$.
\end{enumerate}

\begin{proof}[{\bf Proof of Theorem~\ref{main}}]
The implication (i) $\Rightarrow$ (ii)  follows from Propositions~\ref{compSchur}. The implication (ii) $\Rightarrow$ (iii) is trivial.

(iii) $\Rightarrow$ (i). Fix a dense subgroup $H\subseteq \Lambda(G)$ on which every Schur ultrafilter converges to $1_G$. Let $H_d$ be the group $H$ endowed with the discrete topology, and $i: H_d\rightarrow G$ be the identity homomorphism. By Proposition~\ref{key}, there exists a continuous homomorphism $\phi:\beta H_d\rightarrow G$ such that $\phi(x)=x$ for every $x\in H$. Note that $\phi$ is surjective, because $H$ is dense in $G$.  
Fix any Schur ultrafilter $u$ on $H$. Since $u$ converges to $1_G$, we get that $V\cap H\in u$ for each open neighborhood $V\subseteq G$ of $1_G$. By the continuity of the map $\phi$, for each open neighborhood $W\subseteq G$ of $\phi(u)$ there exists $U\in u$ such that $\phi(U)=U\subseteq W$. Since the space $G$ is Hausdorff, we get $\phi(u)=1_G$. Proposition~\ref{autom} implies that $G$ is a topological group.
\end{proof}

The following result follows from~\cite[Example 8.2]{FU}, see also~\cite[Exercise 1.3.40B]{Book}.

\begin{proposition}\label{example}
There exists a chart group $G$ which is not a topological group.
\end{proposition}

Theorem~\ref{main} and Proposition~\ref{example} imply that Proposition~\ref{compSchur} does not generalize over chart groups. Namely,  
there exist a chart group $G$ and a Schur ultrafilter $u$ on $\Lambda(G)$ such that $u$ does not converge to $1_G$.  The following result implies that there exists a discrete group $G$ such that the congruences $\Xi$ and $\Theta$ on $\beta G$ are distinct.  

\begin{proposition}
There exists a discrete group $H$ and a closed congruence $\rho$ on $\beta H$ such that $\rho$ merges to $1_H$ all idempotents, but not all Schur ultrafilters.   
\end{proposition}

\begin{proof}
Let $G$ be a chart group which is not a topological group (see Proposition~\ref{example}). Let $H$ be the group $\Lambda(G)$ endowed with the discrete topology. By Proposition~\ref{key}, the identity map $i:H\rightarrow G$ extends to a continuous homomorphism $\phi:\beta H\rightarrow G$, which is surjective, as $i(H)=H$ is dense in $G$. 
Since $G$ is a chart group, $\rho=\Ker(\phi)$ is a closed group congruence on $\beta H$. It follows that $\Theta\subseteq \rho$, where $\Theta$ is the congruence on $\beta H$ from Definition~\ref{def}. As $G$ is not a topological group, Proposition~\ref{autom} implies the existence of a Schur ultrafilter $u$ on $H$ such that $\phi(u)\neq 1_G$, implying $(u,1_H)\notin \rho$.     
\end{proof}


\section*{Acknowledgments}
We are grateful to the referee for numerous constructive remarks that essentially improved
the paper.

\end{document}